\documentclass[preprint]{elsarticle}
\usepackage{amssymb}           % AMS fonts
\usepackage{bm}                % font for vector ($\bm{v}_1$)

\usepackage{amsmath}           % need for subequations
\usepackage{amsthm}            % need for theorem environments
\usepackage{amscd}             % need for commutative diagrams
\usepackage{extarrows}         %%% need for \xlongrightarrow[down]{up}, etc.

\usepackage{ascmac}            % need for "screen", "itembox", etc.
\usepackage{indentfirst}       % need for indention of 1st paragraph

\usepackage{graphicx}          % need for figures
\newtheorem{theorem}{Theorem}[section]
\newtheorem{prop}[theorem]{Proposition}
\newtheorem{cor}[theorem]{Corollary}

\theoremstyle{definition}

\newtheorem{example}[theorem]{Example}
\newtheorem{remark}[theorem]{Remark}
\newcommand{\R}{\ensuremath{\mathbb{R}}}

\newcommand{\Z}{\ensuremath{\mathbb{Z}}}

\newcommand{\Col}{\mathop{\mathrm{Col}}\nolimits}
\newcommand{\s}{\mathop{\mathcal{S}}\nolimits}

\begin{document}

%%%%%%%%%%%%%%%%%%%%%%%%%
% Subject classification 
%%%%%%%%%%%%%%%%%%%%%%%%%

%\renewcommand{\thefootnote}{\alph{footnote}}
%\footnote[0]{2000 {\it Mathematics Subject Classification}.
%Primary 57Q45; Secondary 57M25.}

%\subjclass[2000]{Primary 57Q45; Secondary 57M25.}
\date{\today}
%\keywords{Surface-knot, Rack, Branch point, Quandle}

%%%%%%%%%
% Title
%%%%%%%%%

\title{On rack colorings for surface-knot diagrams without branch points}
%\title[Rack colorings for surface-knots]{On rack colorings 
%for surface-knot diagrams without branch points}

%%%%%%%%%%%%%%%%%%%%%%%%%%%%%%
% Author names and addresses 
%%%%%%%%%%%%%%%%%%%%%%%%%%%%%%

\author{Kanako Oshiro}
\address{Department of Information and Communication Sciences, 
Sophia University, 7-1 Kioicho, Chiyoda-ku, Tokyo 102-8554, Japan}
\ead{oshirok@sophia.ac.jp}
%\email{oshirok@sophia.ac.jp}

\author{Kokoro Tanaka}
\address{Department of Mathematics, Tokyo Gakugei University, 
Nukuikita 4-1-1, Koganei, Tokyo 184-8501, Japan}
%Nukuikita-machi 4-1-1, Koganei-shi, Tokyo 184-8501, Japan}
\ead{kotanaka@u-gakugei.ac.jp}
%\email{kotanaka@u-gakugei.ac.jp}

%\author{Kokoro TANAKA}
%\address{Department of Mathematics, 
%Gakushuin University, 1-5-1 Mejiro, Toshima-ku, 
%Tokyo 171-8588, Japan}
%\email{tanaka@math.gakushuin.ac.jp}

%\author{Kokoro Tanaka}
%\address{Graduate School of Mathematical Sciences, 
%University of Tokyo, 3-8-1 Komaba Meguro, 
%Tokyo 153-8914, Japan}
%\email{k-tanaka@ms.u-tokyo.ac.jp}

%%%%%%%%%%%%%
% Dedication
%%%%%%%%%%%%%

%\dedicatory{}

%%%%%%%%%%%%%
% Abstract 
%%%%%%%%%%%%%
\begin{abstract}
%It is known that 
Racks do not give us invariants of surface-knots in general. 
For example, if a surface-knot diagram has branch points 
(and a rack which we use satisfies some mild condition), 
%(and a rack is a non-quandle connected rack), 
then it admits no rack colorings. 
In this paper, we investigate rack colorings for surface-knot diagrams 
without branch points and prove that 
rack colorings are invariants of $S^2$-knots. 
We also prove that rack colorings for $S^2$-knots 
can be interpreted in terms of quandles, 
%Rack colorings for $S^2$-knot diagrams without branch points 
%can be interpreted by the associated quandle of a rack, 
%which implies that rack colorings are invariants of $S^2$-knots. 
and discuss a relationship with regular-equivalences of surface-knot diagrams.  
%where two diagrams representing the same surface-knot are said to be 
%regular-equivalent if they are related by a finite sequence of 
%\lq\lq branch-free\rq\rq\ Roseman moves. 
\end{abstract}


%%%%%%%%%%%%%%%%%%%%%%%%%%%%%%%%%%%%%%%%%%%%%%%%%%%%%%%%%%%%%%%%%%%%%%%%%
% end Topmatter
%%%%%%%%%%%%%%%%%%%%%%%%%%%%%%%%%%%%%%%%%%%%%%%%%%%%%%%%%%%%%%%%%%%%%%%%%
\maketitle

%%%%%%%%%%%%%%%%%%%%%%%%%%%%%%%%%%%%%%%%%%%%%%%%%%%%%%%%%%%%%%%%%%%%%%%%%
% body of paper
%%%%%%%%%%%%%%%%%%%%%%%%%%%%%%%%%%%%%%%%%%%%%%%%%%%%%%%%%%%%%%%%%%%%%%%%%

%\newpage
%%%%%%%%%%%%%%%%%%%%%%%%%%%%%%%%%%%%%
\section{Introduction}\label{sec:intro}

A \textit{quandle} \cite{Joy-82,Mat-82} is an algebraic system 
with three axioms which correspond to the Reidemeister moves, and 
is useful for studying classical knots in the $3$-space. 
For example, for a given quandle, we have an invariant of oriented knots, 
called a \textit{quandle coloring}, 
which has been extensively studied by many researchers 
(cf. \cite{CESY,HHO-12,Inoue-01}). 
There is an algebraic system, called a \textit{rack} \cite{FR-92}, similar to a quandle. 
%There is a similar algebraic system, called a \textit{rack}. 
It has two axioms which correspond to the framed Reidemeister moves, and 
is useful for studying framed classical knots in the $3$-space. 
We note that a quandle is a rack by definition. 
For example, for a given rack, 
we have an invariant of oriented framed knots, 
called a \textit{rack coloring}. 
Rack colorings themselves are not invariants of (unframed) oriented knots. 
However, by using rack colorings, Nelson \cite{Nel-pre} constructed 
an invariant of oriented knots. 
Later, the second author and Taniguchi \cite{TanakaT} gave 
interpretation of his invariant in terms of quandles. 

Quandles are also useful for studying surface-knots in the $4$-space. 
For example, for a given quandle, 
we also have an invariant of surface-knots, called 
a \textit{quandle coloring}. 
A quandle coloring for surface-knots has also been extensively studied 
by many researchers (cf. \cite{Iwa-06,SS-05,Tan-05}). 
On the contrary, as far as the authors know, 
there are no study of surface-knots by using rack theory at present. 
(We note that framed surface-knots, more generally 
framed submanifolds of the $n$-space of codimension $2$ with $n \geq 3$, 
were studied in \cite{FRS-07} by using rack theory.) 
One of reasons is that surface-knot diagrams may have branch points. 
For example, for a given rack, 
we can also define a \textit{rack coloring} for surface-knot diagrams 
in a way similar to quandle colorings. 
However, if a surface-knot diagram has branch points 
(and a rack which we use satisfies some mild condition), 
then it admits no rack coloring. 
We also observe in Example~\ref{ex:Satoh} that 
even if surface-knot diagrams which we consider have no branch points, 
rack colorings are not invariants of surface-knots in general. 

In this paper, we investigate rack colorings for surface-knot diagrams 
without branch points and prove that 
rack colorings are invariants of $S^2$-knots (Theorem~\ref{thm:main1}). 
%even if a sequence of Roseman moves between 
%$D_1$ and $D_2$ representing the same $S^2$-knot may involve branch points 
We also prove that 
rack colorings for $S^2$-knot diagrams without branch points can be interpreted 
in terms of quandles (Theorem~\ref{thm:main2}). 
%in terms of the associated quandle of a rack 
We note that Theorem~\ref{thm:main2} implies Theorem~\ref{thm:main1}. 
In the final section, we also discuss a relationship with 
regular-equivalences of surface-knot diagrams, 
where two diagrams representing the same surface-knot are said to be 
regular-equivalent if they are related by a finite sequence of 
\lq\lq branch-free\rq\rq\ Roseman moves. 

This paper is organized as follows.
Section~\ref{sec:def} is devoted to reviewing racks, quandles, 
surface-knots and their diagrams. 
Our main results, Theorem~\ref{thm:main1} and \ref{thm:main2}, are stated 
in Section~\ref{sec:main}. 
%about rack colorings of $S^2$-knots and their interpretation 
%in terms of quandles. (Theorems~\ref{thm:main1} and \ref{thm:main2}). 
We study rack colorings of 
surface-knot diagrams with immersed curves in Section~\ref{sec:immersed}. 
Section~\ref{sec:associated} provides a relationship 
between a rack coloring with immersed curves and a quandle coloring 
%for surface-knot diagrams by its associated quandle, 
when surface-knot diagrams which we consider have no branch points.  
These two sections are devoted to giving the proof of Theorem~\ref{thm:main2}, 
which is proven in the end of Section~\ref{sec:associated}. 
As mentioned above, Theorem~\ref{thm:main2} implies Theorem~\ref{thm:main1}. 
In Section~\ref{sec:reg-eq}, we discuss a relationship with regular-equivalences 
of surface-knot diagrams.

%\newpage
%%%%%%%%%%%%%%%%%%%%%%%%%%%%%%%%%%%%%
\section{Definition}\label{sec:def}

\subsection{Racks and quandles}
For a non-empty set $X$ and a binary operation $*$ on $X$, 
%$* : X \times X \to X, \ (a,b) \mapsto a*b$, 
%$* : X \times X \to X$, 
we consider the following three conditions. 
\begin{enumerate}
\item[(Q1)]
For any $a \in X$, $a*a=a$. 
\item[(Q2)]
For any $a \in X$, the map $*a : X \to X$, defined by $\bullet \mapsto \bullet *a$, 
is bijective. 
\item[(Q3)]
For any $a,b,c \in X$, $(a*b)*c=(a*c)*(b*c)$. 
\end{enumerate}
These three correspond to the Reidemeister 
moves of type I, II and III respectively. 

A pair $(X,*)$ is called a \textit{quandle} 
if it satisfies conditions (Q1), (Q2) and (Q3). 
Quandles are useful for studying oriented knots and 
also for oriented surface-knots. 
A pair $(X,*)$ is called a \textit{rack} 
if it satisfies conditions (Q2) and (Q3). 
Racks are useful for studying oriented framed knots. 
We remark that a quandle is a rack by definition and 
that a rack/quandle $(X,*)$ is often abbreviated to $X$. 
Racks and quandles have been studied in, for example, \cite{FR-92,Joy-82,Mat-82}.

%For racks $X$ and $Y$, a {\it rack homomorphism} $f:X \to Y$ is a 
%map such that $f(a*b)=f(a)*f(b)$ for any $a,b \in X$. 
%If both $X$ and $Y$ are quandles, we call it a {\it quandle homomorphism}. 

%\bigskip 
%$Aut(X,*)$??, $Inn(X,*)$??

\subsection{Surface-knots and their diagrams}
A \textit{surface-knot} (or a \textit{$\Sigma^2$-knot}) is 
a submanifold of the $4$-space $\R^4$, 
homeomorphic to a closed connected oriented surface $\Sigma^2$. 
%We also use the term \textit{$2$-knot} 
%when $\Sigma^2$ is the $2$-sphere. 
We always assume that all surface-knots are oriented in this paper. 
%underlying closed surfaces are oriented and surface-knots are also oriented. 
Two surface-knots are said to be \textit{equivalent} if they can be deformed into  
each other through an isotopy of $\R^4$.

A \textit{diagram} of a surface-knot is its image 
via a generic projection from $\R^4$ to $\R^3$, 
equipped with the height information as follows: 
At a neighborhood of each double point, there are 
intersecting two disks and 
one is higher than the other with respect to the $4$th coordinate 
dropped by the projection. 
Then the height information is indicated by removing 
the regular neighborhood of the double point in the lower disk 
along the double point curves. 
%Along each double point curve, one sheet is higher than the other 
%with respect to the $4$th coordinate dropped by the projection. 
%Then the height information is indicated by removing 
%the regular neighborhood of the double point curve in the lower sheet. 
Then a diagram is regarded as a disjoint union of connected compact 
oriented surfaces, each of which is called a \textit{sheet}. 
A diagram is basically composed of four kinds of 
local pictures, 
%shown in Figure~\ref{fig:diagram}, 
each of which is the image of a neighborhood of 
a typical point --- 
a regular point, a \textit{double point}, 
an isolated \textit{triple point} or an isolated \textit{branch point}. 
The latter three are depicted in Figure~\ref{fig:diagram}.

\begin{figure}[thbp]\centering
\includegraphics[width=0.7\textwidth]{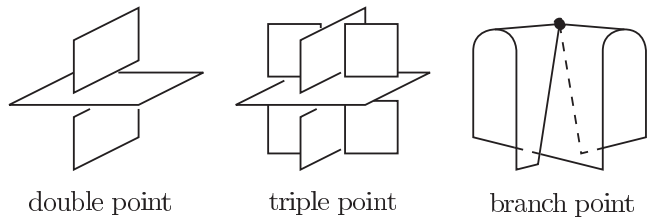}
\caption{Local picture of the projection image of a surface-knot}\label{fig:diagram}
\end{figure}

Two surface-knot diagrams are said to be \textit{equivalent} 
if they are related by (ambient isotopies of $\R^3$ and) 
a finite sequence of seven Roseman moves, 
shown in Figure~\ref{fig:roseman}, where we omit height information for simplicity. 
According to Roseman \cite{Ros-95}, two surface-knots are equivalent if and only if 
they have equivalent diagrams. We refer to \cite{CS-book} for more details.

\begin{figure}[thbp]\centering
%\includegraphics[width=0.5\textwidth]{.eps}
%\caption{Local picture of the image of a surface-knot by a projection}\label{fig:}
%\medskip
\includegraphics[width=1.0\textwidth]{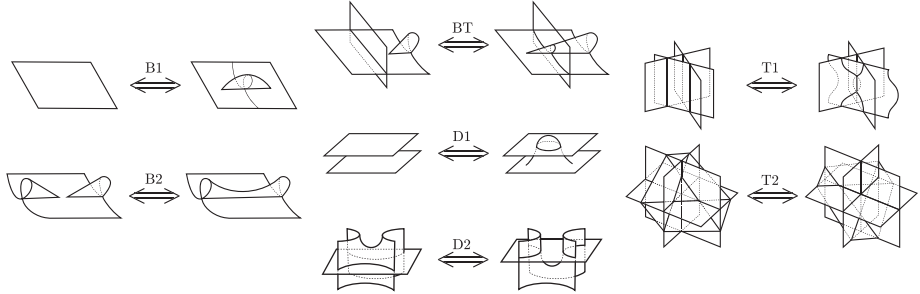}
\caption{Roseman moves}\label{fig:roseman}
\end{figure}

%\newpage
%%%%%%%%%%%%%%%%%%%%%%%%%%%%%%%%%%%%%
\section{Rack colorings and main theorem}\label{sec:main}

\subsection{Rack colorings}
%Let $D$ be a surface-knot diagram.  
We represent the orientation of a surface-knot diagram 
by assigning normal directions $\vec{n}$, 
depicted by an arrow 
which looks like the symbol \lq\lq $\Uparrow$\rq\rq\ 
as in Figure~\ref{fig:coloring-1}, to each sheet of the diagram 
such that the triple $(\vec{v}_1, \vec{v}_2, \vec{n})$ matches 
the orientation of $\R^3$, where the pair $(\vec{v}_1, \vec{v}_2)$ 
denote the orientation of the sheet. 
%Let $\s(D)$ denote the set of all sheet of $D$.

For a surface-knot diagram $D$, let $\s(D)$ denote the set of all sheet of $D$. 
For a rack $R$, a map $c: \s(D) \to R$ is 
a \textit{rack coloring} %(or $R$-coloring)
if it satisfies the following relation along each double point curve. 
Let $x_j$ be the over-sheet along a double point curve, and 
$x_i, x_k$ be under-sheets along the double point curve 
such that the normal direction of $x_j$ points from $x_i$ to $x_k$. 
Then it is required that $c(x_k)=c(x_i)*c(x_j)$ as in Figure~\ref{fig:coloring-1}.  
Let $\Col_R(D)$ be the set of rack colorings of $D$ by $R$. 
We note that when a rack $R$ is finite, $\Col_R(D)$ is also finite.  
%if a rack $R$ is finite then $\Col_R(D)$ is also finite. 
If a rack which we consider is a quandle, 
then a rack coloring is also called a \textit{quandle coloring}. 
%
%A rack coloring of a surface-knot diagram $D$ by a rack $R$ is 
%a map from the set of all sheets of $D$ to $R$ 
% %an assignment of an element of $X$ to each sheet 
%such that $a*b = c$ holds along each double point curve, 
%where $a$ (resp. $c$) is the color of under-sheet that is behind (resp. in front of) 
%the over-sheet colored $b$ with respect to the normal vector of the over-sheet. 

\begin{figure}[thbp]\centering
\begin{minipage}[]{0.35\hsize}
\includegraphics[width=\hsize]{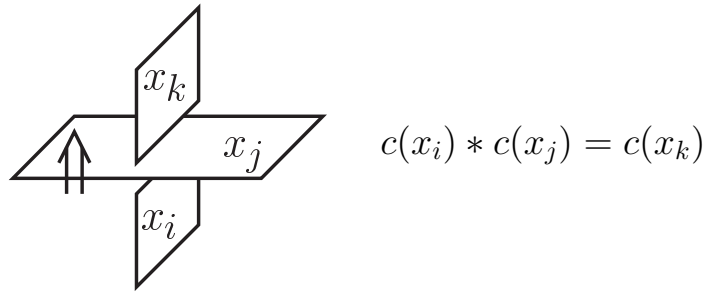}
\end{minipage}\qquad{\Large $c(x_i) * c(x_j) = c(x_k)$}
\caption{Coloring relation along a double point curve}
\label{fig:coloring-1}
\end{figure}

It is known that a quandle coloring by a quandle $Q$ 
is an invariant of surface-knots (cf. \cite{CKS-book}). 
Precisely speaking, for two diagrams $D_1$ and $D_2$ of a surface-knot, 
there is a bijection between $\Col_Q(D_1)$ and $\Col_Q(D_2)$. 
For rack colorings, this is not the case in general. 
One of the difficulties is the existence of branch points. 
%of surface-knot diagrams. 
%
For example, if a diagram $D$ has branch points and 
a rack $R$ has no element $a$ such that $a*a = a$, 
then $D$ admits no rack coloring by $R$, that is, $\Col_R(D)$ is the emptyset. 
This is because the equation $a*a =a$ for some $a \in R$ is required 
as the rack coloring condition along a double point curve 
one of whose endpoints is a branch point. 
(We note that a non-quandle connected rack, such as a cyclic rack 
\cite[Example 7]{FR-92}, has no element $a$ such that $a*a = a$. 
Since we do not use the precise definition of the connectedness for racks, 
we omit the details and give some comments instead. 
It is known that any rack is decomposed into the connected components 
each of which is also a rack. 
It is also known 
that any rack coloring for a surface-knot diagram by a rack 
is essentially 
the one by a connected component of the rack. 
Hence the connectedness for racks is 
%essential 
suitable in considering rack colorings for surface-knot diagrams.)
%If a finite rack $R$ is not a quandle, then the number of the set 
%$\{a \in R \mid a*a=a \}$ is strictly less than the number of the set $R$. 
%
%In the next subsection, 
%we consider a surface-knot diagram without branch points. 

\subsection{Diagrams without branch points}\label{subsec:branch}
Even if surface-knot diagrams which we consider have no branch points, 
we observe, in Example~\ref{ex:Satoh} below,  
that rack colorings are not invariants of surface-knots in general. 
We note that every (oriented) surface-knot has a diagram 
without branch points (cf. \cite{CS-92}). 
%It is known in \cite{CS-92} 
%that every (oriented) surface-knot has a diagram without branch points. 
%, since it has a trivial normal bundle (\cite{CS-92}). 

\begin{example}\label{ex:Satoh}
Let $D_1$ and $D_2$ be two surface-knot diagrams as in Figure~\ref{fig:satoh}. 
Both represent the same $T^2$-knot and do not have branch points. 
(We note that both represent a trivial $T^2$-knot, 
which bounds a solid torus in $\R^4$.) 
However, for a finite non-quandle connected rack $R$, 
the number of $\Col_R(D_1)$ is strictly more than that of $\Col_R(D_2)$, 
since $R$ has no element $a$ such that $a*a=a$ 
which is required as a rack coloring condition around 
the double point curve of $D_2$. 
That is, the number of $\Col_R(D_2)$ is zero, 
while the number of $\Col_R(D_1)$ is equal to that of $R$. 
This is because Roseman moves of type B1 and B2 may change the 
number of rack colorings by a finite rack in general. 
We note that these diagrams are appeared in Satoh's paper \cite{Sat-01}.
We return to this example in Section~\ref{sec:reg-eq} again.

\begin{figure}[thbp]\centering
{\large $D_1$:} \ 
\begin{minipage}{0.3\hsize}
\includegraphics[width=\hsize]{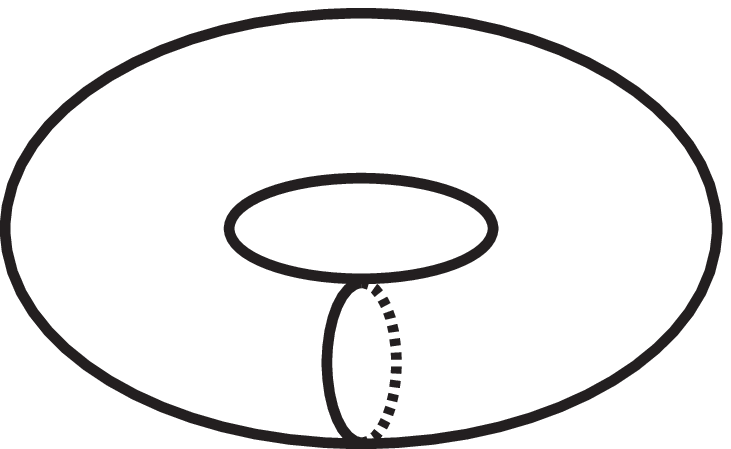}
\end{minipage} \qquad 
{\large $D_2$:} \ 
\begin{minipage}{0.33\hsize}
\includegraphics[width=\hsize]{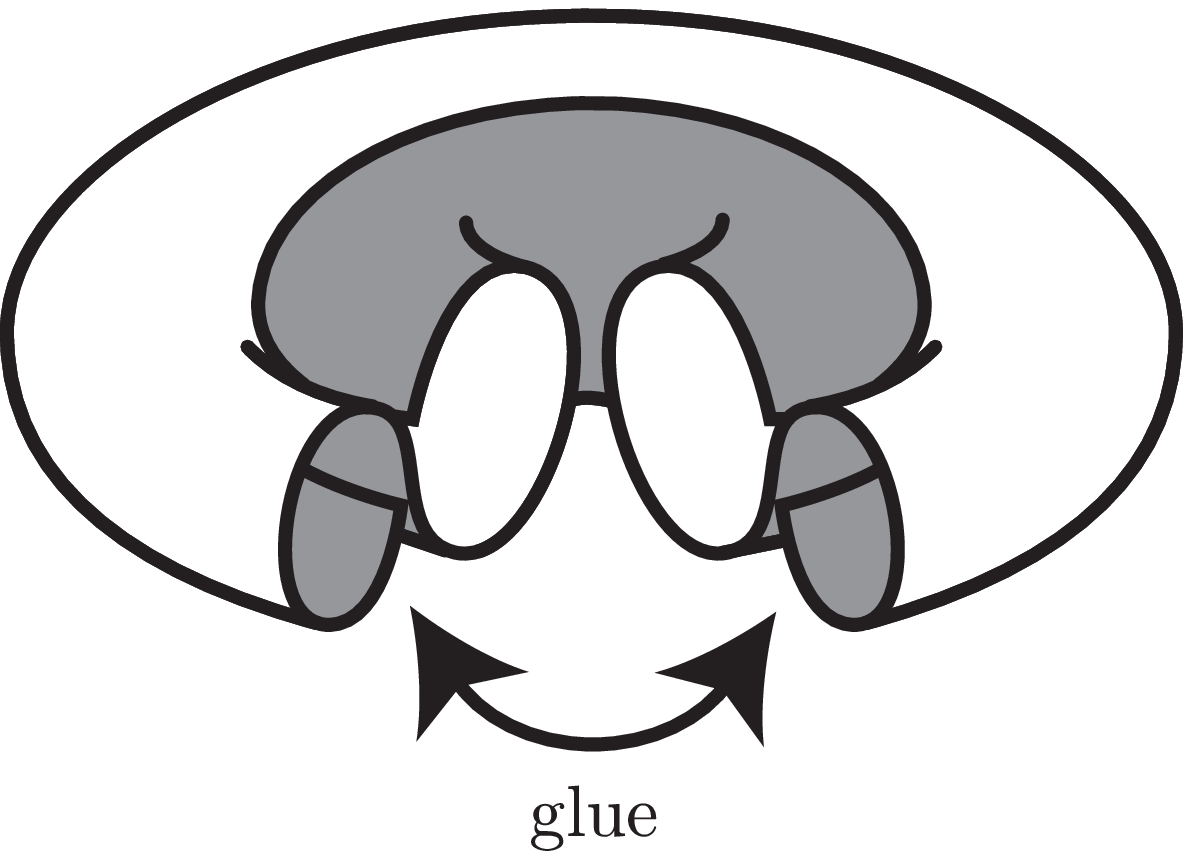}
\end{minipage}
\caption{Satoh's $T^2$-knot diagrams}\label{fig:satoh}
\end{figure}
\end{example}

\subsection{Main theorem}
The above situation is totally different for $S^2$-knots. 
In fact, we can prove the following:  

\begin{theorem}\label{thm:main1}
For any two diagrams $D_1$ and $D_2$ 
of an $S^2$-knot without branch points, 
%the number of $\Col_R(D_1)$ is equal to that of $\Col_R(D_2)$. 
there is a bijection between $\Col_R(D_1)$ and $\Col_R(D_2)$. 
\end{theorem}

This theorem says that a rack coloring is an invariant of $S^2$-knots, 
even if a sequence of Roseman moves between 
$D_1$ and $D_2$ %representing the same $S^2$-knot
may involve branch points. 
For any rack $R$, we will define the associated quandle, denoted by $Q_R$, 
of $R$ in Section~\ref{sec:associated}. 
Then a rack coloring by $R$ can be interpreted in terms of 
its associated quandle $Q_R$ as follows: 

\begin{theorem}\label{thm:main2}
For any $S^2$-knot diagram $D$ without branch points, 
%the number of $\Col_R(D)$ is equal to that of $\Col_{Q_R}(D)$. 
there is a bijection between $\Col_R(D)$ and $\Col_{Q_R}(D)$. 
\end{theorem}

Since quandle colorings are invariants of surface-knots, 
Theorem~\ref{thm:main2} implies Theorem~\ref{thm:main1}. 
Hence, in what follows, 
we focus on proving Theorem~\ref{thm:main2}. 
Before that, 
we will show Theorem~\ref{thm:immersed} in Section~\ref{sec:immersed} 
and Theorem~\ref{thm:associated} in Section~\ref{sec:associated}. 
Combining these two theorems, 
we will prove Theorem~\ref{thm:main2} in 
the end of Section~\ref{sec:associated}.
%Section~\ref{sec:main2}.

%\newpage
%%%%%%%%%%%%%%%%%%%%%%%%%%%%%%%%%%%%%
\section{Rack colorings with immersed curves}\label{sec:immersed}

\subsection{The kink map of a rack}
For a rack $R=(R,*)$, let $\iota : R \to R$
%, denote it by $\iota$ for simplicity if there is no confusion, 
be the map, called the (negative) \textit{kink map} \cite{TanakaT} of $R$, 
characterized by the equation $\iota (a)*a=a$ for any element $a \in R$. 
The map $\iota$ is well-defined by the condition (Q2). 
It might be better to denote it by a symbol like \lq\lq $\iota_R$\rq\rq , 
since this map is uniquely determined by $R$. 
However, we denote it by $\iota$ for simplicity. 
We note that (a notion similar to) the map $\iota$ 
has essentially appeared in \cite{Nel-pre}. 
See Remark~\ref{rem:kink} below for comments on 
a hidden diagrammatic meaning of the kink map $\iota$. 
It is known in \cite{TanakaT} that the kink map $\iota$ satisfies 
the following three conditions. 
\begin{enumerate}
\item[(K1)]
The map $\iota$ is bijective.
\item[(K2)] 
For any $a,b \in R$, $\iota(a)*b=\iota(a*b)$. 
\item[(K3)]
For any $a,b \in R$, $a * \iota(b)= a*b$. 
\end{enumerate}
By the condition (K1), there is a unique inverse map, 
denoted by $\iota^{-1}$ in a usual way, of the kink map $\iota$. 
%and we denoted it by $\iota^{-1}$ in a usual way. 
Then, for any integer $n$, a symbol $\iota^n$ does make sense, 
that is, the symbol $\iota^n$ denote
the $n$-times composition of $\iota$ for $n \geq 0$ and 
the $|n|$-times composition of $\iota^{-1}$ for $n < 0$. 

\begin{remark}\label{rem:kink}
%The kink map $\iota$ of a rack $R$ has a hidden geometric meaning. 
%We remark that 
The kink map $\iota$ of a rack $R$ 
corresponds to a negative kink of a classical knot diagram 
in the following sense. 
The negative kink consists of the two arcs as in Figure~\ref{fig:kink}. 
If we assign an element $a \in R$ to the arc which comes in, 
then the arc which goes out receives the element $\iota(a) \in R$ 
by the rack coloring condition at the crossing, 
where rack colorings for classical knots 
are defined in a way similar to that for surface-knots and 
%can be defined in a way similar to that for surface-knots and 
we omit the details for the definition. 
%we omit the details for precise definition. 
Thus we can think of the kink map $\iota$ as 
an algebraic abstraction of the negative kink.  

\begin{figure}[thbp]\centering
\begin{minipage}[]{0.3\hsize}
\includegraphics[width=\hsize]{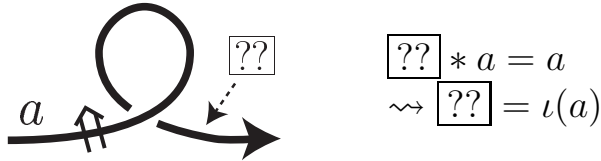}
\end{minipage}\qquad \qquad 
\begin{minipage}[]{0.4\hsize}
{\Large $\fbox{??} * a = a$ \\ 
\quad \quad $\rightsquigarrow$ $\fbox{??}=\iota(a)$}
\end{minipage}
\caption{The map kink $\iota$ and a negative kink}\label{fig:kink}
\end{figure}
\end{remark}

\subsection{Rack colorings with immersed curves}
Consider a set $L$ of oriented immersed curves on a surface-knot diagram $D$. 
We assume that $L$ intersects itself transversely 
and each multiple point is a double point.
We further assume that $L$ intersects the double point curves of $D$ 
transversely, and misses triple points and branch points of $D$. 
The orientation for each immersed curve of $L$ is represented 
by normal directions $\vec{n}$, 
depicted by an arrow 
which looks like the symbol \lq\lq $\uparrow$\rq\rq\  
as in Figure~\ref{fig:curve}, 
such that the pair $(\vec{v}, \vec{n})$ matches the orientation of $D$, 
where $\vec{v}$ denote the orientation of the immersed curve. 
%Consider the set $L$ of oriented immersed curves on $D$ each of which intersects 
%the double point curves transversely and misses triple points (and branch points), 
%and each of whose multiple points is a transverse double point. 
%each of which intersects the double point curves transversely, 
%and misses triple points (and branch points). 
%
%A neighborhood of a double point of $D$ looks like Figure~\ref{fig:} for example. 

\begin{figure}[thbp]\centering
\includegraphics[width=0.35\textwidth]{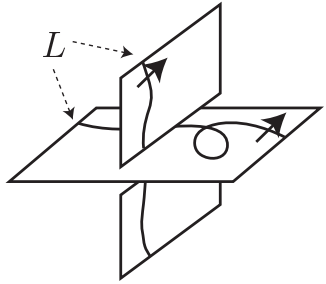}
\caption{Local picture of oriented immersed curves on a diagram}\label{fig:curve}
\end{figure}

By cutting each sheet in $\s(D)$ along the curves of $L$, 
we obtain a disjoint union of connected compact oriented 
surfaces, each of which is also called a \textit{sheet} of the pair $(D,L)$. 
We note that the local picture of a diagram in Figure~\ref{fig:curve} 
consists of seven sheets. 
We denote by $\s(D,L)$ the set of all sheet of $(D,L)$. 
%Each connected component of $D$ by cutting a sheet in $\s(D)$ along 
%the curves of $L$ is also called a \textit{sheet} of the pair $(D,L)$, 
%and we denote by $\s(D,L)$ the set of all sheet of $(D,L)$. 
%
We note that each sheet in $\s(D,L)$ is included in a unique sheet 
in $\s(D)$ as a subset. 
For a rack $R$, a map $c: \s(D,L) \to R$ is a 
\textit{rack coloring with immersed curves} %($R$-coloring with immersed curves)
if it satisfies the two relation: 
%the following two relation. 
\begin{itemize}
\item
One is the same relation along each double point curve of $D$ as 
that in (usual) rack colorings, and 
\item
The other is the following relation along each oriented immersed curve of $L$. 
Let $x_i$, $x_j$ be two sheets along an oriented immersed curve 
such that the normal direction of the curve points from 
$x_i$ to $x_j$. Then it is required that $c(x_j)=\iota(c(x_i))$ 
as in the left of Figure~\ref{fig:coloring-2}. 
\end{itemize}
\begin{figure}[thbp]\centering
$\begin{minipage}[]{0.3\hsize}
\includegraphics[width=\hsize]{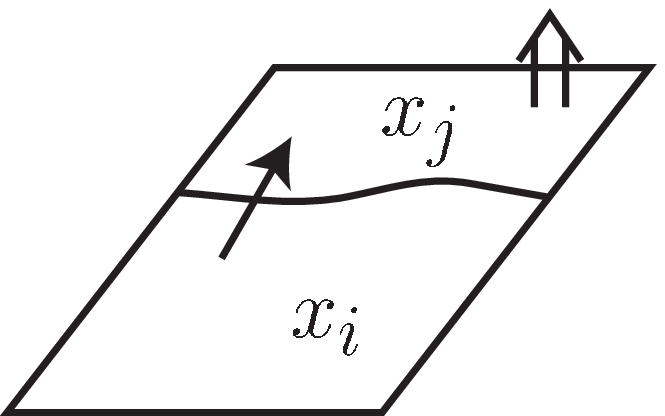}

\medskip{\Large $c(x_j)=\iota(c(x_i))$}
\end{minipage}\quad
\left( \quad 
\xlongleftrightarrow[\textrm{\large Remark~\ref{rem:immersed}}]{\phantom{xxxxxxxx}} 
\begin{minipage}[]{0.3\hsize}
\includegraphics[width=\hsize]{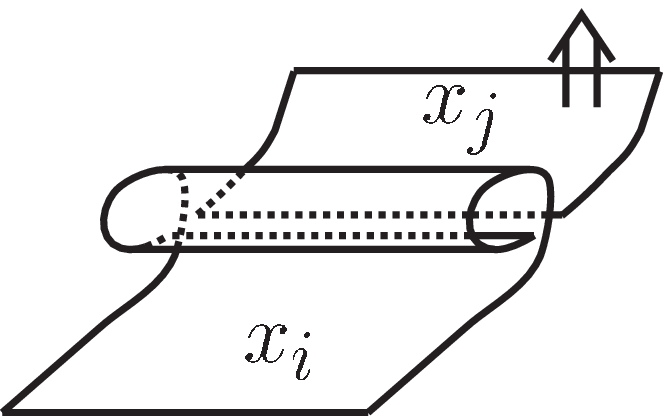}

\medskip{\Large $c(x_j)=\iota(c(x_i))$}
\end{minipage} \ \right) $
\caption{Coloring relation along an oriented immersed curve}
\label{fig:coloring-2}
\end{figure}
See Remark~\ref{rem:immersed} below for comments on the latter 
(somewhat artificial) coloring condition. 
Let $\Col_R(D,L)$ be the set of rack colorings of $(D,L)$ by $R$. 
We note that when a rack $R$ is finite, $\Col_R(D,L)$ is also finite. 
%
%We briefly discuss the well-definedness of rack coloring with immersed curves. 
At each double point, say $p$, of $L$ and double point curves of $D$, 
the well-definedness of rack coloring with immersed curves is not so obvious. 
By using the properties (K2) and (K3) of the kink map $\iota$, 
we can show the following near the double point $p$. 
\begin{itemize}
\item
If an arc of $L$ lies on the under-sheets, 
the condition (K2) ensures the well-definedness at $p$. 
The situation is illustrated on 
the left of Figure~\ref{fig:well-defined}. 
\item
If an arc of $L$ lies on the over-sheet, 
the condition (K3) ensures the well-definedness at $p$. 
The situation is illustrated on 
the right of Figure~\ref{fig:well-defined}.  
\end{itemize}

\begin{figure}[thbp]\centering
\includegraphics[width=0.83\hsize]{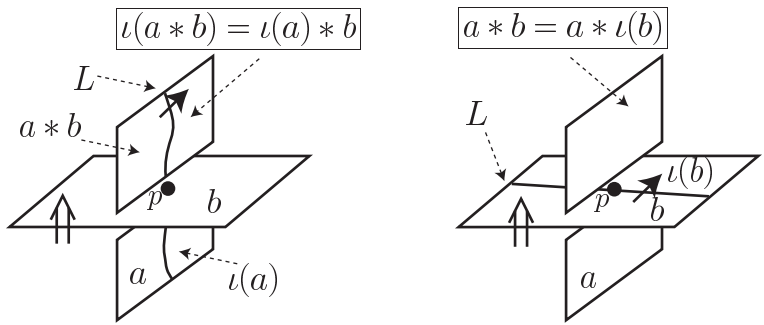}
\caption{Well-definedness of rack colorings with immersed curves}
\label{fig:well-defined}
\end{figure}

\begin{remark}\label{rem:immersed}
We mention here a hidden meaning of a set $L$ of 
oriented immersed curves on a surface-knot diagram $D$. 
A neighborhood of an arc of $L$ on $D$ is virtually 
considered as an abstraction of \lq\lq (kink)$\times [0,1]$\rq\rq\ 
as in the right of Figure~\ref{fig:coloring-2}. 
If we replace the neighborhood of the arc of $L$ virtually 
with \lq\lq (kink)$\times [0,1]$\rq\rq , 
then the (usual) rack coloring condition implies 
our coloring condition along each oriented immersed curve of $L$. 
\end{remark}

\begin{prop}\label{prop:immersed}
%Let $D$ be a surface-knot diagram without branch points. 
For two sets $L_1$ and $L_2$ of oriented immersed curves 
on a surface-knot diagram $D$, 
if $L_1$ is homologous to $L_2$, then 
%the number of $\Col_R(D,L_1)$ is equal to that of $\Col_R(D,L_2)$. 
there is a bijection between $\Col_R(D,L_1)$ and $\Col_R(D,L_2)$. 
\end{prop}

\begin{proof}[Proof of Proposition~\ref{prop:immersed}]
We may assume that $L_1$ and $L_2$, by perturbing $L_1$ or $L_2$ if necessary, 
intersect transversely and each multiple point on the union $L_1 \cup L_2$ 
is a double pont.
We denote by $-L_1$ the set $L_1$ with the opposite orientation. 
Since $L_1$ and $L_2$ are homologous, 
the union $(-L_1) \cup L_2$ is null-homologous on $D$.

We choose a sheet, say $x_0$, of $(D, (-L_1) \cup L_2)$, and 
assign an integer to each sheet of $(D, (-L_1) \cup L_2)$ 
by the following Alexander numbering-like rule: 
\begin{itemize}
\item%[(i)]
We assign $0$ to the sheet $x_0$. 
\item%[(ii)]
For two adjacent sheets $x_1$ and $x_2$ along an arc of $(-L_1) \cup L_2$, 
we suppose that the normal orientation of the arc points from $x_1$. 
Then the integer assigned to $x_2$ is larger than that to $x_1$ by $1$. 
%The difference of the indices assigned to $x_2$ and $x_1$ is equal to $1$. 
\item%[(iii)]
For two adjacent sheets along a double point curve of $D$, 
the integers assigned to them are the same.  
\end{itemize} 
See Figure~\ref{fig:numbering}. 
Since  $(-L_1) \cup L_2$ is null-homologous on $D$, 
this assignment is well-defined.

\begin{figure}[thbp]\centering
\includegraphics[width=0.8\textwidth]{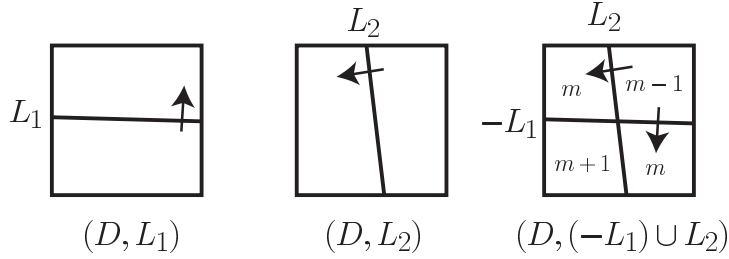}\\
\caption{Numbering on $\s(D, (-L_1) \cup L_2)$}\label{fig:numbering}
\end{figure}

Using the above numbering on $\s(D, (-L_1) \cup L_2)$, 
we construct an explicit bijection $\Phi: \Col_R(D,L_1) \to \Col_R(D,L_2)$.
For a coloring $c \in \Col_R(D,L_1)$, we define $\Phi(c) \in \Col_R(D,L_2)$ as follows: 
Let $x_2$ be a sheet in $\s(D, L_2)$. 
The sheet $x_2$ might be divided into the further smaller 
sheets on $(D, (-L_1) \cup L_2)$. 
(In the case where $L_1 \cap x_2= \emptyset$, the sheet $x_2$ is not divided.) 
For the sheet $x_2$, 
choose one of the small sheets on $(D, (-L_1) \cup L_2)$, and denote it by $y$. 
We note that $y$ is included in $x_2$ as a subset on $(D, L_2)$. 
Let $x_1$ be a unique sheet in $\s(D, -L_1)$ 
such that $x_1$ includes $y$ as a subset on $(D, -L_1)$. 
See Figure~\ref{fig:bijection-phi}. 
Then we define $\Phi(c)(x_2)=\iota^n (c(x_1))$, 
where $n$ is the numbering on the small sheet $y$.

The well-definedness of the map $\Phi$ is shown as follows: 
Two adjacent small sheets $y$ and $y'$, included in $x_2$ as subsets on 
$(D, (-L_1) \cup L_2)$, are separated by an arc, say $\alpha$, of $-L_1$. 
Suppose that  the normal orientation of $\alpha$ points from $y$. 
By the above rule, it holds that $n'=n+1$, where 
the integer $n$ (resp. $n'$) is the numbering on the small sheet $y$ (resp. $y'$). 
Let $x_1$ and $x_1'$ be the sheets in $\s(D, -L_1)$ such that 
$x_1$ (resp. $x_1'$) includes $y$ (resp. $y'$) as a subset on $(D, -L_1)$. 
%$y$ (resp. $y'$) is included in $x_1$ (resp. $x_1'$) as a subset on $(D, -L_1)$. 
See Figure~\ref{fig:bijection-phi}. 
By the rack coloring condition for immersed curves, the equation 
$c(x_1) = \iota (c(x_1'))$ holds. 
%$\iota (c(x_1)) = c(x_1')$ holds. 
Therefore, we have 
%$\iota^{n} (c(x_1)) = \iota^{n'} (c(x_1'))$. 
$$
\iota^{n} \bigl( c(x_1) \bigr) = \iota^{n'-1} \bigl( c(x_1) \bigr) 
= \iota^{n'-1} \Bigl( \iota \bigl( c(x_1') \bigr) \Bigr) 
= \iota^{n'} \bigl( c(x_1') \bigr) .
$$
We can also check that 
the rack coloring conditions along $L_2$ and the double point curves of $D$
on $\s(D, L_2)$ are satisfied. 
The inverse map of $\Phi$ can be constructed in a way similar to the 
definition of $\Phi$. 
(The only difference is to replace $n$ in the definition of $\Phi$ with $-n$.)

\begin{figure}[thbp]\centering
\includegraphics[width=0.80\textwidth]{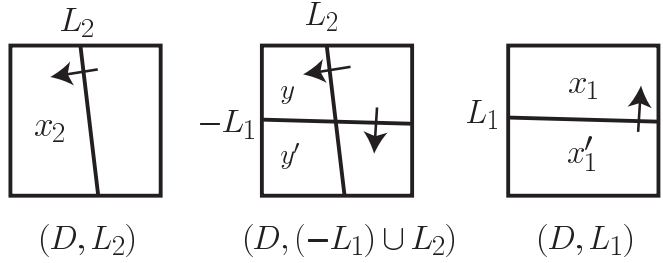}
\caption{Local picture of the bijection $\Phi$}\label{fig:bijection-phi}
\end{figure}
\end{proof}

Since any set of oriented immersed curves on an $S^2$-knot diagram is 
null-homologous, Proposition~\ref{prop:immersed} implies:

\begin{theorem}\label{thm:immersed}
%Let $D$ be an $S^2$-knot diagram without branch points. 
For any set $L$ of oriented immersed curves on an $S^2$-knot diagram $D$, 
%the number of $\Col_R(D,L)$ is equal to that of $\Col_R(D)$. 
there is a bijection between $\Col_R(D,L)$ and $\Col_R(D)$. 
\end{theorem}

%\newpage
%%%%%%%%%%%%%%%%%%%%%%%%%%%%%%%%%%%%%
\section{Associated quandles}\label{sec:associated}

%\subsection{Associated quandles}
For a rack $R=(R,*)$ and the kink map $\iota$ of $R$, 
we define a new binary operation $*^{\iota}$ on the set $R$ 
by $a*^{\iota}b:=\iota(a)*b$. 
Then it is known in \cite{AG-03,TanakaT} that the pair $(R,*^{\iota})$ 
becomes a quandle. 
The quandle $(R, *^{\iota})$ is called the \textit{associated quandle} of $R$ 
and is denoted by $Q_R$. 
%We note that $Q_R$ has been used in \cite{} for the study of 
%rack coloring invariants of oriented classical knots. 
If a surface-knot diagram has no branch points, 
then quandle colorings by $Q_R$ can be interpreted 
in terms of rack colorings with immersed curves by $R$ as follows: 

\begin{theorem}\label{thm:associated}
For any surface-knot diagram $D$ without branch points, 
there exists a set $L$ of oriented immersed curves on $D$ such that 
%the number of $\Col_{R}(D,L)$ is equal to that of $\Col_{Q_R}(D)$.  
there is a bijection between $\Col_{Q_R}(D)$ and $\Col_R(D,L)$. 
\end{theorem}

\begin{proof}[Proof of Theorem~\ref{thm:associated}]
We take a set $L$ of oriented immersed curves as 
a copy of the double point curves of $D$ 
by pushing the double point curves slightly on the under-sheet 
in the direction opposite to the normal direction of the over-sheet 
along each double point curve as in the right of Figure~\ref{fig:bijection-psi}, 
where the normal orientation of $L$ matches that of the over-sheet. 
The set $L$ is well-defined, since the surface-knot diagram $D$, 
which we now consider, has no branch points. 
We note that when $D$ has no triple points (and no branch points), 
the set $L$ consists of oriented embedded curves. 

We construct an explicit bijection $\Psi: \Col_{Q_R}(D) \to \Col_R(D,L)$. 
See Figure~\ref{fig:bijection-psi}, 
where we depict a local picture of $\Psi$ along an arc of the double point curves. 
For a coloring $c \in \Col_{Q_R}(D)$, we define $\Psi(c) \in \Col_R(D,L)$ as follows: 
Let $y$ be a sheet in $\s(D,L)$. 
It follows from the construction of $L$ that 
there exists a unique sheet, say $x$, in $\s(D)$ such that 
$x$ includes $y$ as a subset on $D$. 
%$x'$ is included in $x$ as a subset on $D$. 
Then we define $\Psi(c)(y) := \iota(c(x))$ if the sheet $y$ 
lies between an arc of $L$ and the double point curve 
that is parallel to the arc, and $\Psi(c)(y) := c(x)$ if not. 
This map $\Psi$ is well-defined, since the definition of 
the binary operation of $Q_R$ requires 
the equality $a*^{\iota}b = \iota(a)*b$.

\begin{figure}[thbp]\centering
$$\begin{array}{ccc}
\mathrm{Col}_{Q_R}(D) & 
\xlongrightarrow[1:1]{\phantom{xxxx}{\text{\Large $\Psi$}}\phantom{xxxx}} & 
\mathrm{Col}_{R}(D,L) \\
\begin{minipage}[]{0.35\hsize}
\includegraphics[width=\hsize]{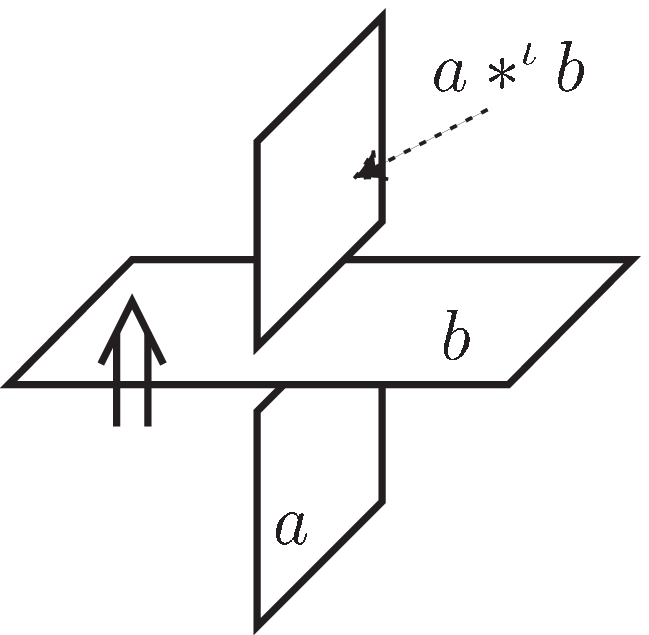}
\end{minipage} 
& %\begin{array}{c} a*^{\iota} b \\ \parallel \\  \iota(a) * b \end{array}
& 
\begin{minipage}[]{0.35\hsize}
\includegraphics[width=\hsize]{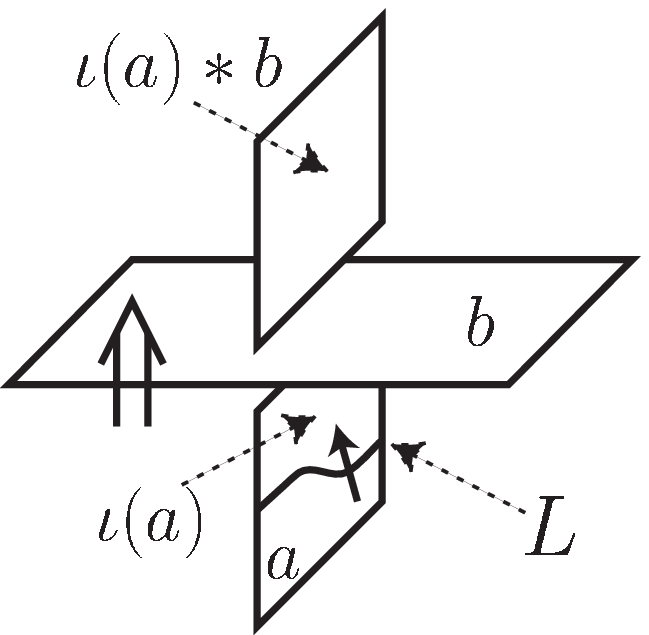}
\end{minipage}
\end{array}$$
\caption{Local picture of the bijection $\Psi$}\label{fig:bijection-psi}
\end{figure}

To prove the bijectiveness of $\Psi$, 
we construct a map $\Psi': \Col_R(D,L) \to \Col_{Q_R}(D)$ and 
show that $\Psi'$ is the inverse map of $\Psi$.  
For a coloring $c' \in \Col_R(D,L)$, 
we define $\Psi'(c') \in \Col_{Q_R}(D)$ as follows: 
Let $x$ be a sheet in $\s(D)$. 
It follows from the construction of $L$ that 
there exists a unique sheet, say $y$, in $\s(D,L)$ such that 
$y$ is included in $x$ as a subset on $D$ and 
does not lie between an arc of $L$ and the double point curve 
that is parallel to the arc.  
Then we define $\Psi'(c')(x) := c'(y)$. 
Again the definition of the binary operation of $Q_R$ 
ensures the well-definedness of $\Psi'$ and 
we can check that $\Psi'$ is the inverse map of $\Psi$ 
by an explicit calculation. 
\end{proof}

%%%%%%%%%%%%%%%%%%%%%%%%%%%%%%%%%%%%%
%\section{Proof of Theorem~\ref{thm:main2}}\label{sec:main2}
%In this short section, we prove Theorem~\ref{thm:main2}. 
%by using Theorem~\ref{thm:immersed} 
%and Theorem~\ref{thm:associated}. 
%

With Theorem~\ref{thm:immersed} and Theorem~\ref{thm:associated} in hand, 
we now prove Theorem~\ref{thm:main2}.

\begin{proof}[Proof of Theorem~\ref{thm:main2}]
For an $S^2$-knot diagram $D$ without branch points, 
by using Theorem~\ref{thm:associated}, 
there exists a set $L$ of oriented immersed curves on $D$ such that 
there is a bijection between $\Col_{Q_R}(D)$ and $\Col_{R}(D,L)$. 
Then, by using Theorem~\ref{thm:immersed}, 
there is a bijection between $\Col_{R}(D,L)$ and $\Col_{R}(D)$. 
By the composition of these two bijections, we have a desired bijection. 
%Composing these two bijections, we have the conclusion. 
\end{proof}

%\newpage
%%%%%%%%%%%%%%%%%%%%%%%%%%%%%%%%%%%%%
\section{Regular-equivalences and rack colorings}\label{sec:reg-eq}

Finally we discuss a relationship with  
regular-equivalences of surface-knot diagrams. 
As we mentioned in Subsection~\ref{subsec:branch}, 
any (oriented) surface-knot has a diagram without branch points. 
In \cite{Sat-01}, 
such a \lq\lq branch-free\rq\rq\ diagram is called a \textit{regular} diagram, 
and two regular diagrams are said to be 
\textit{regular-equivalent} if they are related by a finite 
sequence of \lq\lq branch-free\rq\rq\ Roseman moves, 
that is, the moves of type $D1$, $D2$, $T1$ and $T2$ 
in Figure~\ref{fig:roseman}. 
%that is, the Roseman moves of type $D1$, $D2$, $T1$ and $T2$. 
%in Figure~\ref{fig:roseman}. 
(For surface-knot diagrams, which may have branch points, 
we can also define the notion of \lq\lq regular-equivalence\rq\rq\ 
in a similar way. 
That is, two surface-knot diagrams are said to be 
\textit{regular-equivalent} if they are related by a finite 
sequence of \lq\lq branch-free\rq\rq\ Roseman moves.)
We can easily check that 
there is a one-to-one correspondence between rack colorings 
before and after each \lq\lq branch-free\rq\rq\ Roseman move. 
(We note that there is also a one-to-one correspondence 
before and after the Roseman move of type $BT$.)
Then we have the following: 

\begin{theorem}\label{thm:reg-eq}
For two diagrams $D_1$ and $D_2$ of a surface-knot, 
if they are regular-equivalent, then 
%the number of $\Col_R(D_1)$ is equal to that of $\Col_R(D_2)$. 
there is a bijection between $\Col_R(D_1)$ and $\Col_R(D_2)$. 
\end{theorem}

This theorem says that 
rack colorings are invariants of regular-equivalence classes of 
%surface-knot diagrams. 
diagrams of a surface-knot. 
Using a contraposition of Theorem~\ref{thm:reg-eq}, 
we can reprove that Satoh's examples are not regular-equivalent,  
that is, Example~\ref{ex:Satoh} implies the following: 

\begin{cor}\textrm{$($\cite[Theorem 3]{Sat-01}$)$}
There exist two regular surface-knot diagrams 
which are equivalent but not regular-equivalent. 
%
%Let $D_1$ be a $T^2$-knot diagram as in Figure~\ref{} and 
%$D_2$ a a $T^2$-knot diagram as in Figure~\ref{}. 
%They both represent the same $T^2$-knot, but 
%they are not regular-equivalent. 
\end{cor}

%We remark here that we cannot say 
%We remark here that 
Theorem~\ref{thm:reg-eq} cannot tell us 
anything about regular-equivalence classes of diagrams of an $S^2$-knot. 
Precisely speaking, 
for two diagrams $D_1$ and $D_2$ of an $S^2$-knot, 
it follows from Theorem~\ref{thm:main1} 
%our main theorem (Theorem~\ref{thm:main1}) 
that even if they are not regular-equivalent, 
%the number of $\Col_R(D_1)$ must be equal to that of $\Col_R(D_2)$. 
there is a bijection between 
the two sets, 
$\Col_R(D_1)$ and $\Col_R(D_2)$, 
%of rack colorings 
of rack colorings by a rack $R$. 
However, Takase and the second author \cite{TakaseT} proved that, 
for any regular $S^2$-knot diagram $D$,  
there exists a regular diagram $D'$ such that 
$D$ and $D'$ represent the same $S^2$-knot but they are not regular-equivalent. 
They used immersion theory to prove 
such a theorem 
instead of rack theory.

%\newpage
%%%%%%%%%%%%%%%%%%%
% Acknowledgments
%%%%%%%%%%%%%%%%%%

\section*{Acknowledgments}
The authors thank Sam Nelson for several helpful comments.
The first author is partially supported by
Grant-in-Aid for Young Scientists (B) (No.~25800052), 
Japan Society for the Promotion of Science. 
The second author is partially supported by 
Grant-in-Aid for Scientific Research (C) (No.~26400082), 
Japan Society for the Promotion of Science. 

%The first author has been supported in part by 
%
%Grant-in-Aid for Young Scientists (B) (No.~XXXXXXXX), 
%The Ministry of Education, Culture, Sports, Science and Technology, Japan.
%
%Grant-in-Aid for Scientific Research (C) (No.~XXXXXXXX), 
%Japan Society for the Promotion of Science.

%%%%%%%%%%%%%%%%
% bibliography
%%%%%%%%%%%%%%%
%\bibliographystyle{amsplain}

%%%%%%%%%%%%%%%%%%%
% Appendix
%%%%%%%%%%%%%%%%%%

%\appendix

%%%%%%%%%%%%%%%%%%%%%%%%%%%%%%%%%%%%%%%%%%%%%%%%%%%%%%%%%%%%%%%%%%%%%%%%%
% End of the document
%%%%%%%%%%%%%%%%%%%%%%%%%%%%%%%%%%%%%%%%%%%%%%%%%%%%%%%%%%%%%%%%%%%%%%%%%
\end{document}